\documentclass[11pt]{amsart}
\usepackage{amsmath,amsthm,amsfonts,latexsym,amssymb}
\usepackage{graphicx}
\usepackage{float}

\calclayout\evensidemargin .1in

\newtheorem{theorem}{Theorem}[section]
\newtheorem{lemma}[theorem]{Lemma}
\newtheorem{proposition}[theorem]{Proposition}

\theoremstyle{definition}

\newtheorem{example}[theorem]{Example}
\newtheorem{remark}[theorem]{Remark}

\newtheorem{definition}[theorem]{Definition}

\makeatletter
\@namedef{subjclassname@2020}{%
	\textup{2020} Mathematics Subject Classification}
\makeatother

\title[Compatible Relative Lefschetz Fibrations on Admissible Relative Stein Pairs]{Compatible Relative Lefschetz Fibrations on \\ Admissible Relative Stein Pairs}

\author{Yasemin Yildirim}
\address{Department of Mathematics, TED University, Ankara, TURKEY}
\email{yyildirim@tedu.edu.tr}

\author{M. Firat Arikan}
\address{Department of Mathematics, Middle East Technical University, Ankara, TURKEY}
\email{farikan@metu.edu.tr}


\begin{document}
	
\begin{abstract}
For more than two decades it has been known that any compact Stein surface (of real dimension four) admits a compatible Lefschetz fibration over a two-disk. More recently, Giroux and Pardon have generalized this result by giving a complex geometric proof for the existence of compatible Lefschetz fibrations on Stein domains of any even dimension. As a preparatory step in proving the former, Akbulut and Ozbagci have shown that there exist infinitely many pairwise non-equivalent Lefschetz fibrations on the four-ball $D^4$ by using a result of Lyon constructing fibrations on the complements of $(p,q)$-torus links in the three-sphere. In this paper, we first extend this result to obtain compatible Lefschetz fibrations on $D^6$ whose pages are $(p,q,2)$-Brieskorn varieties, and then construct a compatible ``relative'' Lefschetz fibrations on any Stein domain (of dimension six) which admit a certain (``admissible'') ``relative Stein pair'' structure.  In particular, we provide a purely topological proof for the existence of Lefschetz fibrations on specific $6$-dimensional Stein domains.
\end{abstract}
	
\maketitle
	

\section{Introduction}
A \textit{Stein manifold } is a properly embedded complex submanifold  of some $ \mathbb{C}^N $. A properly embedded complex submanifold of a Stein manifold is called a \textit{Stein submanifold}. A Stein (sub)manifold is necessarily an open manifold and even dimensional. In this paper, we are interested in only the compact case for which the following equivalent characterization is given:\\
	 	
A \textit{Stein domain } $(W,J,\phi)$ is a compact complex manifold of real dimension $2n$ with boundary which admits a complex structure $J$ and a $J$-convex generalized Morse function $\phi:W\rightarrow \mathbb{R}$. A properly embedded complex submanifold $X \subset W$ equipped with the restricted structures $(J|_X,\phi|_X)$ is called a \textit{Stein subdomain}. Eliashberg has  shown that Stein domains carry handle decompositions up to the \emph{critical index} $n$ \cite{eli}.\\

A \textit{Weinstein domain} $(W,\omega,\phi)$ is a compact exact symplectic manifold $(W,\omega)$ that carries a Morse function $\phi:W\rightarrow \mathbb{R}$ such that the Liouville vector field of $\omega$ is gradient-like for $\phi$ and transversally pointing out from $\partial W$. A properly embedded submanifold $X \subset W$ equipped with the restricted structures $(\omega|_X,\phi|_X)$ is called a \textit{Weinstein subdomain}. Any Stein domain admits a Weinstein structure, and the converse has been proved by Cieliebak and Eliashberg \cite{ciel}.\\

A \textit{contact manifold} $(M^{2n-1},\xi)$ is a smooth manifold $M$ equipped with a  completely non-integrable hyperplane distribution $\xi$. A $k$-dimensional ($k\leq n-1$) submanifold $L$ of $(M^{2n-1},\xi)$ is \textit{isotropic} if $T_pL\subset \xi_p$ for all $p \in L$, and an $(n-1)$-dimensional isotropic submanifold is said to be \textit{Legendrian}. Stein (resp. Weinstein) domains carry contact structures on their boundaries induced by Stein (resp. Weinstein) structures, and the handles which construct Stein (resp. Weinstein)  domains are attached along isotropic spheres. In what follows, we will consider \textit{the standard contact sphere} ${S}^{2n-1}$ equipped with \textit{ the standard contact structure}
\begin{center}
$\xi^{2n-1}_{std}=\textrm{Ker}(\alpha^{2n-1}_{std}), \quad \alpha^{2n-1}_{std}=\displaystyle\sum_{i=0}^{n-1}x_idy_i-y_idx_i$
\end{center}
where ${S}^{2n-1}$ is considered to be the unit sphere in $\mathbb{R}^{2n}$ with coordinates $(x_0,y_0,...,x_{n-1},y_{n-1})$.

An \textit{open book} on a closed manifold $M$ is a pair $(B,f)$ where $B$ is a codimension $2$-submanifold of $M$ (called the \textit{binding}) and the map $f:M-B\rightarrow S^1$ is a fibration such that $\partial \overline{f^{-1}(\theta)} =B$ for all $\theta \in S^1$. The closure of the fiber $F= \overline{f^{-1}(\theta)}$ is called the \textit{page} over $\theta \in \mathbb{S}^1$, and the \textit{monodromy} of $(B,f)$ is the self-diffeomorphism $f:F\to F$ defined by the return map of the flow of the vector field $\partial/\partial \theta$. An open book $(B, f)$ and the contact structure $\xi$ on $M$ are \textit{compatible} if there exists a contact form $\alpha$ for $\xi$ such that $\alpha$ is a contact form on $B$, $d\alpha$ defines a symplectic structure on each page, and the orientation on $B$ induced from $\alpha$ agrees with the one induced from $d\alpha$ considering $B$ as the contact boundary of a symplectic page.\\

Ever since  Donaldson showed the existence of compatible Lefschetz fibrations on compact symplectic manifolds \cite{donald}, the theory of Lefschetz fibrations is well-treated in several different settings. Here we merely focus on those fibering over $D^2$ and whose total spaces are Weinstein domains. We refer the reader to \cite{donald}, \cite{duffs} and  \cite{sei} for more comprehensive definitions and facts about Lefschetz fibrations in different settings.\\
	 	
	Let $E$ be a compact $2n$-dimensional manifold with corners equipped with an exact symplectic form $\omega=d\lambda$ such that the both faces of the  boundary $\partial E=\partial_v E\cup \partial_h E$ are convex.
A \textit{compatible Lefschetz fibration} on a compact $2n$-dimensional manifold with corners $E$ is a holomorphic map $ \pi: E^{2n}\rightarrow D^2$ for some choice of compatible complex structure on $E$, which satisfies the following conditions:

\begin{enumerate}
	\item The map $\pi$ has finitely many nondegenerate distinct critical values $s_1,...,s_k \in int(D^2)$, and there exists an unique critical point $r_j \in  \pi^{-1}(s_j)$ for each $j=1,...,k$,
	and compatible complex structure on $E$ is integrable near the critical points,
	
	\item Near each critical point $r_j$ and the corresponding critical value $s_j$, there are local complex coordinate charts matching with the orientations of $E$ and  $D^2$ such that the map $\pi$ locally has the form $\pi(z_1,...,z_n)=z_1^2+...+z_n^2$,
	
	\item The restriction of the map $\pi$ to $E\backslash \pi^{-1}(\{s_1,...,s_k\})$ is a locally trivial fibration whose fibers are $(2n-2)$-dimensional exact symplectic manifolds with convex boundary for the restricted forms,
	
	\item The restriction of the map $\pi|_{\partial_v E}: \partial_v E\rightarrow \partial D^2$ to the vertical boundary $\partial_v E=\pi^{-1}(\partial D^2)$ is a surjective smooth fiber bundle. Moreover, there is a neighborhood of the horizontal boundary $\partial_h E=\bigcup\limits_{z\in D^2} \partial(\pi^{-1}(z))$ such that the restriction of $\pi$ to this neighborhood is a product fibration $D^2\times \mathcal{N}(\partial F)$ where $F$ is a regular fiber of $\pi$ and $ \mathcal{N}(\partial F)$ is a neighborhood of $\partial F$. The restricted $\omega$ and $\lambda$ to this neighborhood are sum of forms from the two factors. Also, $\pi|_{\partial_h E}: \partial_h E\rightarrow D^2$ are surjective fiber bundles.
\end{enumerate}

The corners of $E$ can be rounded to get a Liouville domain so that one can obtain an exact symplectic Lefschetz fibration on the Liouville domain over disk whose regular fibers are Liouville subdomains equipped with the restriction of $\lambda$. If $E$ has a Weinstein structure endowed with a Lefschetz fibration whose fibers are $(2n-2)$-dimensional Weinstein subdomains, then by rounding the corners of $E$ one can get a compatible Lefschetz fibration $\pi:W\rightarrow D^2$ on the Weinstein domain $W$ such that each regular fiber is a Weinstein subdomain with the restricted structure.

\textit{A vanishing path} is an embedded path $\gamma:[0,1]\rightarrow \mathbb{C}$ that starts in a regular value and ends in a critical value. \textit{A Lefschetz thimble}  $\Delta_\gamma$ can be associated for each vanishing path, which is the unique embedded Lagrangian ball $D^{n}$ that satisfies $\pi(\Delta_\gamma)=\gamma([0,1])$ and $\pi(\partial \Delta_\gamma)=\gamma(0)$. The boundary $L$ of the Lefschetz thimble is called the \textit{vanishing cycle} of the vanishing path $\gamma$. It is known that the monodromy about any critical value of a Lefschetz fibration is described by a Dehn twist $\tau_L$ along the corresponding vanishing cycle. Any compatible Lefschetz fibration on a Stein (or Weinstein) domain induces a compatible open book decomposition on the contact boundary. The fibers over $S^1=\partial D^2$ give the pages of the open book. The \textit{(global) monodromy} of a Lefschetz fibration is defined to be the monodromy of the boundary open book.\\
	 	
In \cite{girpar}, Giroux and Pardon constructed Lefschetz fibrations on Stein (and so Weinstein) domains by using complex geometric techniques and the close relationship between Stein and Weinstein structures given in \cite{ciel}. We give a topological proof of this result in dimension six for certain cases. To this end, we introduce the following relative structure.

\begin{definition} \label{def:Stein_Pair}	 	
A pair $(W^6,X^4)$ is a \textit{relative Stein pair} if $W$ is a Stein domain, $X$ is a Stein subdomain of $W$, and such that the induced pair $(W,X)$ of symplectic manifolds admits a Weinstein structure in which the handles are relative (\emph{relative handle decomposition}) in the following sense:

\begin{enumerate}
\item[(0)] The unique $0$-handle $D^4$ of $X$ is properly embedded in the unique  $0$-handle $D^6$ of $W$.\\
\item The (\emph{relative}) $1$-handles of $W$ and $X$ coincide in the following way:\\

\noindent They have the same number of $1$-handles and the cocore of each $1$-handle  of $X$ is properly embedded in the cocore of the corresponding $1$-handle of $W$.\\

\item Each (\emph{relative}) $2$-handle attachment of $W$ corresponds to a $2$-handle attachment of $X$ in the following way:\\
	 		
\noindent The attaching circle $K$ of each  $2$-handle of $W$ is an embedded Legendrian knot in the contact boundary  $\partial X_1$ so that the framing of the $2$-handle of $X$ is $tb(K) -1$ where $X_1 =D^{4}\cup 1$-handles. Also, the cocore of each $2$-handle of $ X $ is properly embedded in the cocore of the corresponding $2$-handle of $W$. \\

\item Each (\emph{relative}) $3$-handle attachment of $W$ corresponds to a $2$-handle attachment of $X$ in the following way:\\
	 		
\noindent An equator of the attaching sphere of each $3$-handle of $W$ is the attaching circle of the corresponding $2$-handle of $X$. This equator is properly embedded in $\partial X_2 \cap \partial W_2$ where $X_2 =X_1\cup 2$-handles, $W_2=W_1\cup 2$-handles, and $W_1=D^6\cup 1$-handles.
\end{enumerate}
\end{definition}
\vspace{.03cm}	 	
The main idea of this definition is the following: While we create the Stein domain $W$ handle by handle, the Stein domain $X$ is evolving as a subdomain of $W$ in the meantime. The Liouville vector field of $W$ is hence tangent to $X$, with all critical points contained in $X$, and the index of the critical point in $X$ being less than or equal to the index of the critical point in $W$. Moreover, the advantage of the conditions in Definition \ref{def:Stein_Pair} is that ``\emph{admissible Stein pairs}'' (see Definition \ref{def:admissible}), which form a subclass of relative Stein pairs, can be described by means of special diagrams which we call \textit{relative Stein diagrams}. Such diagrams are, indeed, relative versions of standard Stein diagrams of Stein surfaces introduced by Gompf in \cite{Gm}.

\begin{definition}
 A \textit{ relative Stein diagram} of a relative Stein pair $(W,X)$ \textit{in the standard form} (see Figure \ref{fig:Rel_Stein_diagram}) is a diagram defined in the following way:

$\bullet$ The (common) $r$ 1-handles of $W$ and $X$ are shown by $r$ pairs of horizontal balls.

$\bullet $ There is a collection of Legendrian (isotropic) horizontal distinguished dashed line segments coresponding to each pair of ball. A dashed line segment is drawn whenever the isotropic attaching circle of an index $2$-handle of $W$ crosses over the corresponding $1$-handle.

$\bullet $ There is a collection of Legendrian (isotropic) horizontal distinguished solid line segments coresponding to each pair of ball. A solid line segment is drawn whenever the isotropic equatorial circle of the Legendrian attaching sphere of an index $3$-handle of $W$ crosses over the corresponding $1$-handle.

$\bullet$ There is a front projection of a framed \textit{Legendrian (isotropic) tangle} (i.e., disjoint union of framed dashed and solid Legendrian (isotropic) knots and arcs) with endpoints meeting the distinguished dashed and solid segments so that they together form the isotropic attaching circles and the equatorial circles of the Legendrian attaching spheres of index $2$- and $3$-handles of $W$, respectively.  	
\end{definition}

\begin{figure}[h!]
	\centering
	\includegraphics{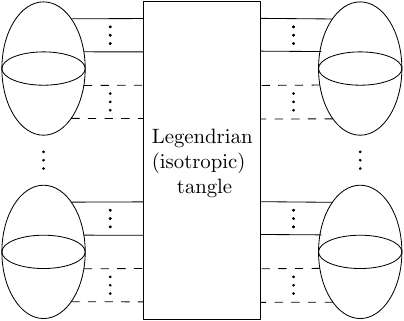}
	\caption{A relative Stein diagram of $(W,X)$ in the standard (Gompf) form.}
	\label{fig:Rel_Stein_diagram}
\end{figure}

\begin{remark}
It will be explained (at the beginning of Section \ref{proof}) that any admissible Stein pair $(W,X)$ can be represented by a relative Stein diagram where the framing of each dashed/solid Legendrian (isotropic) attaching knot is one less than its Thurston-Bennequin framing in the contact boundary of $X_1=D^4 \cup 1$-handles, and hence, a relative Stein diagram in the standard form of the pair $(W,X)$ is, indeed, a standard Stein diagram of the Stein surface $X$. Moreover, by representing $1$-handles by dotted circles (Akbulut convention), $(W,X)$ can be also described by a \textit{relative Stein diagram in dotted form} as depicted in Figure \ref{fig:Dotted_rel_Stein_diagram}.
\end{remark}

\begin{figure}[h!]
	\centering
   \includegraphics{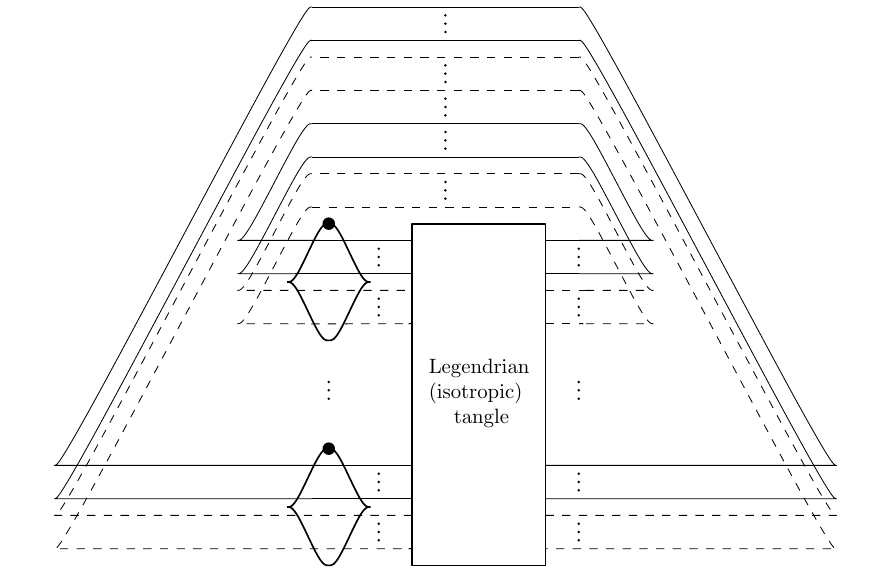}
	\caption{A relative Stein diagram in the dotted (Akbulut) form.}
	\label{fig:Dotted_rel_Stein_diagram}
\end{figure}	

\begin{remark} \label{rmk:Rel_surgery_diag}
In a relative Stein diagram (in dotted form) of $(W,X)$, by replacing dotted unknots with $(+1)$-framed unknots one obtains a ``\textit{relative contact surgery diagram}'' for the boundary \textit{relative contact pair} $(\partial W,\partial X)$ (also see \cite{as}). In particular, the empty relative Stein diagram describes the relative Stein pair $(D^6,D^4)$ and the corresponding empty relative contact surgery diagram describes the boundary relative (\textit{standard}) contact pair $(S^5,S^3)$ equipped with $(\xi^5_{std},\xi^3_{std})$. In the case when $W$ has no index $2$- and $3$-handles (i.e., when there are no dashed and solid knots and arcs), the relative Stein diagram (consisting of $r$ dotted unknots only) describes the relative Stein pair $(\natural_r S^1 \times D^5, \natural_r S^1 \times D^3)$ where $\natural$ denotes the boundary connected sum operation. In that case the corresponding relative contact surgery diagram (consisting of $r$  $(+1)$-framed unknots only) describes the relative contact pair $(\#_r S^1 \times S^4, \#_r S^1 \times S^2)$ equipped with their \textit{standard} (Stein fillable) \textit{contact structures} $(\eta^5_r,\eta^3_r)$.
\end{remark}

Next, we define a class of Legendrian $2$-links  (also see \cite{as}) below where we make use of the conventions $(\#_0 S^1 \times S^4, \#_0 S^1 \times S^2)=(S^5,S^3)$ and $(\eta^5_0,\eta^3_0)=(\xi^5_{std},\xi^3_{std})$.

\begin{definition}	 	
Let $r\geq 0$ be any fixed integer. A \textit{relative Legendrian $2$-link} $(L,K)$ in the relative contact pair $(\#_r S^1 \times S^4, \#_r S^1 \times S^2)$ equipped with $(\eta^5_r,\eta^3_r)$ is a link of Legendrian $2$-spheres $L$ in $(\#_r S^1 \times S^4, \eta^5_r)$ whose isotropic equatorial link $K$ is a Legendrian $1$-link in $(\#_r S^1 \times S^2, \eta^3_r)\subset (\#_r S^1 \times S^4, \eta^5_r)$.
\end{definition}

\begin{remark}
By drawing its Legendrian (isotropic) equatorial $1$-link $K$, any relative Legendrian $2$-link $(L,K)$ can be depicted in the relative contact surgery diagram of $(\#_r S^1 \times S^4, \#_r S^1 \times S^2)$ mentioned in Remark \ref{rmk:Rel_surgery_diag}. (Just introduce a suitably drawn (solid) knot to the diagram which realizes the front projection of each component of $K$.)
\end{remark}

In order to state our main result (see Theorem \ref{thm:Main_theorem} below), one also needs two other notions introduced below. First, for a given relative Stein pair $ (W^6,X^4) $, one can consider a Lefschetz fibration on $W$ consistent with the one on $X$ as follows:
	
\begin{definition} 	
A \textit{compatible relative Lefschetz fibration} on a relative Stein pair $(W^6,X^4)$ is a pair $(\pi_{W},\pi_{X})$ where $\pi_{W}:W\rightarrow D^2$, $ \pi_{X}:X\rightarrow D^2 $ are compatible Lefschetz fibrations  with regular fibers $F_W^4$ and $F_X^2$, respectively, such that
\begin{enumerate}
\item $F_{X}$ is a properly embedded subdomain of $F_{W}$,
\item The set of critical values of $ \pi_{W} $ is contained in the set of critical values of $\pi_{X}$.
\end{enumerate}
\end{definition}

The main interest of the present paper is the following subclass of (relative) Stein pairs:

\begin{definition}\label{def:admissible}
	A pair $(W,X)$ is  called  an \textit{admissible Stein pair} if one has
	\begin{center}
	$W=\widetilde{W}\cup (\sqcup_{j=1}^{\ell} H_j)$ \quad and \quad $X=\widetilde{X}\cup (\sqcup_{j=1}^{\ell} h_j)$
	\end{center}
	 where $\widetilde{W}$ is a subcritical Stein $6$-domain which splits as $\widetilde{W}=\widetilde{X}\times D^2$, and for each $j=1,...,\ell$ the pair $(H_j,h_j)$ is a relative $3$-handle attached along a \textit{suspendible} Legendrian $2$-link.
\end{definition}

\begin{remark}\label{rmk:admissible_relative}
\textit{Suspendibility} of $2$-links will be introduced in Definition \ref{def:suspendible_link} and Definition \ref{def:suspendible_link_general}. For an admissible Stein pair $(W,X)$, its subcritical part $ (\widetilde{W},\widetilde{X}) $ satisfies the conditions $ (0) $, $ (1) $ and $ (2) $ of Definition \ref{def:Stein_Pair} (see the proof of Proposition \ref{mainprop2} or Remark \ref{rmk:subcritical}). Also by definition of admissibility, the condition $ (3) $ of Definition \ref{def:Stein_Pair} is obviously satisfied. Hence, after attaching relative $ 3 $-handles to the relative Stein pair $ (\widetilde{W},\widetilde{X}) $ we obtain a new relative Stein pair $ (W,X) $. Therefore, admissible Stein pairs, indeed, form a special subclass of relative Stein pairs.
\end{remark}
	 	
Finally, we state our main result:
		
\begin{theorem} \label{thm:Main_theorem}
Every admissible Stein pair $(W^6,X^4)$ admits a compatible relative Lefschetz fibration $(\pi_{W},\pi_{X})$.
\end{theorem}
	 	
In \cite{ao}, Akbulut and Ozbagci constructed compatible Lefschetz fibrations on Stein surfaces by using  torus knots and handlebody theory. In order to construct Lefschetz fibrations, they also used Eliashberg's characterization of compact Stein surfaces \cite{eli} and examined Stein surfaces case by case with respect to their handle decompositions. For completeness we recall their approach below:
	 	
First, consider the case when a given Stein surface $X$ has no $1$-handles. Consider a Legendrian link $L$ with $\ell$ components in the contact boundary $S^3=\partial D^4$ which is formed by the attaching circles of the Stein $2$-handles of $X$. Put $L$ into its square bridge position and apply Lyon's algorithm to realize each component $L_i$ of $L$ ($i=1,...,\ell$) on a Seifert surface of a $(p,q)$-torus link (a page of a compatible open book on $(S^3,\xi^3_{std})$). Then by construction the framing $tb(L_i)-1$ of each Stein $2$-handle coincides with the page framing $lk(L_i,L_i^+)-1$, and hence, each Stein $2$-handle of $X$ is indeed a Lefschetz $2$-handle.
For the general case, considering the Stein $1$-handles of $X$ as dotted circles,  insert them into the previous case and realize them on the Seifert surface constructed as well. Finally, observe that attaching a Stein $1$-handle is equivalent to removing a disk from each fiber (near boundary) and extending the monodromy with identity.
	 	
The main goal of this paper is to generalize (extend) the above techniques and ideas of Akbulut and Ozbagci to dimension $6$. In particular, we show that $D^6$ carries a compatible Lefschetz fibration with fibers $V_\epsilon(p,q,2)$ in Section \ref{bireskornfillings}. On the way, we also prove in Section \ref{lyon} that one can realize any relative Legendrian $2$-link in some Brieskorn variety $V_\epsilon(p,q,2)$ which can be considered as a higher dimensional version of Lyon's result.

	
\section{BRIESKORN VARIETIES AS SEIFERT FILLINGS} \label{bireskornfillings}
	 	
Higher dimensional analogues of torus knots are Brieskorn manifolds. The \textit{Brieskorn varieties} are affine varieties of the form
\begin{center}
$V_\epsilon(a_0,...,a_n)=\big\{(z_0,...,z_n)\in \mathbb{C}^{n+1}| \sum\limits_{j=0}^{n} z_j^{a_j}=\epsilon \big\}$.
\end{center}
The links of singular Brieskorn varieties at 0, that is, the sets of the form
\begin{center}
$\Sigma(a_0,...,a_n) =\big\{(z_0,...,z_n)\in \mathbb{C}^{n+1}\;|\; \sum\limits_{j=0}^{n} |z_j|^2=1 $\quad and \quad  $\sum\limits_{j=0}^{n} z_j^{a_j}=0\big\},$
\end{center}
are called \textit{Brieskorn manifolds}. Equivalently, Brieskorn manifolds $\Sigma(a_0,...,a_n)$  arise as the intersection $V_\epsilon(a_0,...,a_n)\cap S^{2n+1}$. On a Brieskorn manifold $\Sigma(a_0,...,a_n)$, an open book decomposition can be explicitly constructed as follows: Consider the map
\begin{center}
$\Psi: \Sigma(a_0,...,a_n) \rightarrow \mathbb{C}$,
\hspace*{1\baselineskip} $(z_0,...,z_n)\longmapsto z_n$.
\end{center}
Normalization of this map gives an open book with binding $B=\{z_n=0\}= \Sigma(a_0,...,a_{n-1})$ and page $V_\epsilon (a_0,...,a_{n-1})$.  (See \cite{kn} and \cite{kk} for more details.) In particular, on the Brieskorn manifold $\Sigma(p,q,2,1)$ (which is diffeomorphic to $S^5$), one can construct an open book by projecting onto the last coordinate. Thus, we obtain an open book of $S^5$ with binding  $\Sigma(p,q,2)$ and pages $V_\epsilon(p,q,2)$.
	 	
In $S^3$, every link bounds an orientable surface which is called a \textit{Seifert surface}. By definition, the binding of an open book decomposition of $S^3$ is a fibered link and each page is a Seifert surface of the link. By adapting this terminology to fibered submanifolds of $S^5$, the Brieskorn variety $V_\epsilon (a_0,...,a_{n-1})$ will be called a \textit{Seifert filling} of the Brieskorn manifold $\Sigma(a_0,...,a_{n-1})$. Note that $V_\epsilon(p,q,2)$ is a Seifert filling of $\Sigma(p,q,2)$.
	 	
In \cite{kk}, the monodromy of an open book which arises from a fibered Brieskorn manifold has been studied. In particular, the monodromy of the open book on $S^5=\Sigma(p,q,2,1)$ (discussed above) is the product of right-handed Dehn twists along certain $2$-spheres embedded in the Seifert filling $V_\epsilon(p,q,2)$. One can realize this open book as the boundary open book induced  from a Lefschetz fibration on $D^6$. More precisely, we have
	 	
\begin{proposition}\label{prop2.1}
For each $(p,q)$,  there exists a compatible Lefschetz fibration on $D^6$ whose regular fibers are $V_\epsilon(p,q,2)$. Moreover, the monodromy consists of the product of $(p-1)(q-1)$ right-handed Dehn twists along embedded Lagrangian $2$-spheres in $V_\epsilon(p,q,2)$ intersecting each other in an explicit pattern. In particular, the binding of the induced open book on $S^5$ is $\Sigma(p,q,2)$.
\end{proposition}

\begin{remark}
The explicit intersection pattern of the vanishing cycles (i.e., Lagrangian $2$-spheres in the above statement) are, in fact, the generalization of the explicit pattern described for compatible Lefschetz fibrations on $D^4$ given in \cite{ao}. As we will explain in the proof of the proposition, such patterns are described by plumbing diagrams as we depict in Figure \ref{fig:Intersection_pattern} for three particular cases.
\end{remark}
	 	 	
\begin{figure}[H]
   \centering
   \includegraphics[scale=1.4]{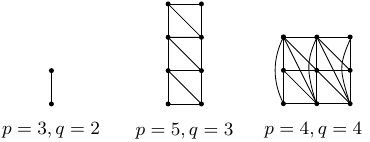}
   \caption{Intersection patterns for some particular cases.}
   \label{fig:Intersection_pattern}
\end{figure}
	 	
\begin{proof} [Proof of Proposition \ref{prop2.1}]
We begin with a compatible Lefschetz fibration on $D^4$ whose monodromy is determined by the torus link $\Sigma(p,q)$, and then extend it to a compatible Lefschetz fibration on $D^6$ whose monodromy is determined by the Brieskorn manifold $\Sigma(p,q,2)$. Denote by $(z_0, z_1)$ the complex coordinates on $\mathbb{C}^2$. The restriction of the map $$\pi: \mathbb{C}^2 \rightarrow \mathbb{C}, \quad \pi(z_{0},z_{1})=z_0^p+z_1^q +M_\delta(z_0,z_1)$$ defines a compatible Lefschetz fibration on $D^4\subset \mathbb{C}^2$. Here, $ M_\delta(z_0,z_1) $ is a morsification of the polynomial $ z_0^p+z_1^q $ where $\delta=(\delta_0,\delta_1)$ for sufficiently small $\delta_0,\delta_1$. (Note that the Hessian of the polynomial $ z_0^p+z_1^q$ is singular if $ p>2$ or $q>2 $ or both. However, Lefschetz fibrations have nondegenerate critical points by definition, so one needs to morsify the polynomial $ z_0^p+z_1^q $.) We will consider the morsification $M_\delta(z_0,z_1)=-\delta_0 z_0-\delta_1z_1$ so that for $\delta_0$ and $\delta_1$ sufficiently small, one can easily check that the above map $\pi: D^4 \rightarrow D^2\subset \mathbb{C}$ has nondegenerate critical points.
	 		
The map $\pi$ has exactly $(p-1)(q-1)$ many nondegenerate distinct critical points, and is injective on the set of its critical points, and so, there are $(p-1)(q-1)$ many vanishing cycles one for each critical value of $\pi$. For a regular value $0<\epsilon \in \mathbb{R}\subset \mathbb{C}$, the regular fiber $F= \pi^{-1}(\epsilon)$ is a Seifert surface $V_\epsilon (p,q)$ of the $(p,q)$-torus link $\Sigma(p,q)$. The monodromy $\varphi_{p,q}$ of the Lefschetz fibration is the product of $(p-1)(q-1)$ right-handed Dehn twists along Lagrangian circles \cite{ao}.
	 		
Next we will consider the ``\textit{suspension}'' $\widetilde{\pi}:D^6 \rightarrow D^2$ of $\pi$.  This is not a customary use of the word ``suspension'', but it is the best option to give the idea presented here. However, its customary meaning will be referred when vanishing cycles are considered.
	 		
By considering the coordinates $(z_0,z_1,z_2)$ on $\mathbb{C}^3$, we extend $\pi$ to obtain Lefschetz fibration $$\widetilde{\pi}:D^6\rightarrow D^2, \quad \widetilde{\pi} (z_0,z_1,z_2)=z_0^p+z_1^q+z_2^2+M_\delta(z_0,z_1).$$
The map $\widetilde{\pi}$ has nondegenerate distinct critical points. To see this, note that $\widetilde{\pi}$ is obtained by ``suspending'' $\pi$ by adding the term $z_2^2$, and so the $z_0$ and $z_1$ coordinates of the critical points of the map $\widetilde{\pi}$ are the same as those of $\pi$, and their $z_2$ coordinates are zero. Therefore, the critical points of $\widetilde{\pi}$ are distinct. Also, one can easily check that the determinant of the Hessian of $\widetilde{\pi}$ is nonzero at each critical point.
	 		
The map $\widetilde{\pi}$ is injective on the set of its $(p-1)(q-1)$ nondegenerate distinct critical points, so it has $(p-1)(q-1)$ vanishing cycles corresponding to its critical values. The regular fiber $\widetilde{F}= \widetilde{\pi}^{-1}(\epsilon)$ of $\widetilde{\pi}$ is the Brieskorn variety $V_\epsilon (p,q,2)$. Note that $\pi$ and $\widetilde{\pi}$ have the same morsification term and if we intersect the fiber $\widetilde{F}$ with $z_2=0$ plane, we obtain the fiber $F$. Also, we have $F\setminus \partial F\subset \widetilde{F}\setminus \partial \widetilde{F}$ and $\partial F=\partial \widetilde{F}\cap F$ which means that the  regular fibers of $\pi$ is properly embedded in those of $\widetilde{\pi}$. Moreover,  as discussed above the critical points of $\widetilde{\pi}$ are the same as those of $\pi$ (with additional $z_2$-coordinates). To understand the vanishing cycles of $\widetilde{\pi}$, one can investigate the Lefschetz fibration $$\Pi:\widetilde{F}\to D^2,\quad \Pi(z_0,z_1,z_2)=z_2$$  on the fiber $\widetilde{F}$. For each critical point of $\pi$, there are exactly two critical points of $\Pi$. So, for each  vanishing cycle in $\pi^{-1}(0)$, there are two vanishing cycles in the fiber of $\Pi$. Since these two vanishing cycles coincide, one can form a matching sphere in ${\widetilde{F}} $ by joining the corresponding two critical points of $\Pi$. These matching spheres form a basis of vanishing cycles in $\widetilde{F}$ for the fibration $\widetilde{\pi}$. They can be taken to intersect in the fiber $\widetilde{\pi}^{-1}(0)$ according to the same graph as the vanishing cycles in $\pi^{-1}(0)$, (see \cite{sei}). In conclusion, the vanishing cycles in the fiber $\widetilde{F}$ of  $\widetilde{\pi}$ are these Lagrangian spheres which we call the suspensions of those in the fiber $F$ of $\pi$.  Hence, introducing a $z_2^2$ term extends (``suspends'') the Lefschetz fibration $\pi$ on $D^4$ to the Lefschetz fibration $\widetilde{\pi}$ on $D^6$, and the monodromy $\widetilde{\varphi}_{p,q}$ of $\widetilde{\pi}$ is the monodromy of the fibered Brieskorn manifold $\Sigma(p,q,2)\subset S^5$ which is the product of $(p-1)(q-1)$ right-handed Dehn twists along Lagrangian spheres which are the suspensions of the corresponding vanishing cycles of $\pi$. In particular, note that the binding $\Sigma(p,q)$ of the  induced open book on $S^3$ is properly embedded in the binding  $\Sigma(p,q,2)$ of the induced open book on $S^5$.

Finally, it is known (see \cite{ao}) that the $(p-1)(q-1)$ vanishing cycles of $\pi$ (sitting on $F$) intersects each other according to a certain intersection pattern which can be described by a plumbing diagram as depicted (for several cases) in Figure \ref{fig:Intersection_pattern}. (A bold dot is drawn for each vanishing cycle and a line segment (or an arc) is drawn connecting two dots for each common point of the corresponding cycles.) Since the vanishing cycles of $\widetilde{\pi}$ (sitting on $\widetilde{F}$) are the suspensions of those of $\pi$, they intersect each other in the same way (and at the same points) that their equators intersect each other. Hence, the vanishing cycles of $\widetilde{\pi}$ (there are exactly $(p-1)(q-1)$ of them) intersect each other according to the same intersection pattern described by the same plumbing diagram.	 		
\end{proof}

\begin{remark} \label{rmk:Compatible_Rel_Lefs_fib}
As we realized above, the vanishing cycles (circles) of $\pi$ on $D^4$ are the equators of the vanishing cycles ($2$-spheres) of $\widetilde{\pi}$ on $D^6$. Indeed, the above proof verifies that $(\widetilde{\pi},\pi)$ is a compatible relative Lefschetz fibration on the (trivial) Stein pair $(D^6,D^4)$ equipped with their standard Stein structures. We will write the monodromy of $(\widetilde{\pi},\pi)$ as the pair $(\widetilde{\varphi}_{p,q},\varphi_{p,q})$.
\end{remark}
	 	
\begin{example}
For $p=3, q=2$, consider the map $\pi:\mathbb{C}^2\rightarrow \mathbb{C}$ defined by $$\pi(z_0,z_1)=z_0^3-\dfrac{z_0}{3^5}+z_1^2$$ where we take $M_\delta(z_0,z_1)=-\dfrac{z_0}{3^5}$ ($\delta_1=0$). One can check $\pi$ restricts to a compatible Lefschetz fibration $\pi:D^4 \to D^2$ as follows: By (simultaneously) solving the equations
$${\partial \pi  \over \partial z_0}=0, \quad {\partial \pi  \over \partial z_1}=0,$$
we see that $(-1/ 27,0)$ and $ (1/ 27,0)$  are the critical points of $\pi$. By substituting these into $\pi$, we get the corresponding critical values $-0.00020322105$ and $-0.00010161052$, respectively.
	 		
Then one needs to calculate the determinant of the Hessian matrix at the critical points. The Hessian matrix is given by (see \cite{ebeling} for a complete discussion):	 		
\begin{center}
$\begin{bmatrix}
\dfrac{\strut \partial^2\pi} {\strut \partial z_0 \partial z_0}& \dfrac{\strut \partial^2\pi } {\strut \partial z_0 \partial z_1}  \\
\dfrac{\strut \partial^2\pi }{\strut\partial z_1 \partial z_0}&\dfrac{\strut \partial^2\pi}{\strut \partial z_1 \partial z_1}
\end{bmatrix}$
=
$\begin{bmatrix}
6z_0&0\\
	 0&2
\end{bmatrix}$
\end{center}
which has nonzero determinants at $(\pm1/ 27,0)$. Therefore, $\pi$ has two nondegenerate critical points and two vanishing cycles corresponding to the critical values above. The regular fiber $F=\pi^{-1}(1)=\{z_0^3-\dfrac{z_0}{3^5}+z_1^2=1\}$ is the plumbing of two left-handed Hopf bands $V_1 (3,2)$ ($2$-dimensional $A_2$-Milnor fiber). The monodromy of the Lefschetz fibration is the product of two right-handed Dehn twists along the red and blue curves which is given in Figure \ref{fig:Braid}.
	 		
\begin{center}
\begin{minipage}{1\textwidth}
\begin{figure}[H]
   \includegraphics[scale=0.8]{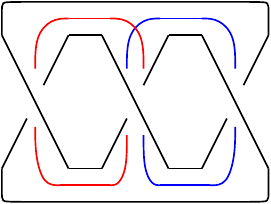}
	\caption{The regular fiber and the monodromy curves of the torus knot $\Sigma(3,2)$ (trefoil).}
	\label{fig:Braid}
	\end{figure}
\end{minipage}\hfill
\end{center}
	 		
Now, we extend (suspend) $\pi$ to a compatible Lefschetz fibration $\widetilde{\pi}$ on $D^6$ which is given by $\widetilde{\pi}(z_0,z_1,z_2)=z_0^3-\dfrac{z_0}{3^5}+z_1^2+z_2^2$. Solving the equations
$${\partial \widetilde{\pi} \over \partial z_0}=0, \quad {\partial \widetilde{\pi}  \over \partial z_1}=0, \quad {\partial \widetilde{\pi}  \over \partial z_2}=0$$
 (simultaneously), we see that $ ({-1/ 27},0,0)$ and $({1/ 27},0,0)$ are the critical points of $\widetilde{\pi}$, and corresponding critical values (coinciding with those of $\pi$) are $-0.00020322105$ and $-0.00010161052$, respectively. The Hessian matrix of $\widetilde{\pi}$ is

\begin{center}
$\begin{bmatrix}
\dfrac{	\strut\partial^2 \widetilde{\pi}} {\strut\partial z_0 \partial z_0}&\dfrac{\strut \partial^2 \widetilde{\pi}}{\strut\partial z_0 \partial z_1}&  \dfrac{\strut\partial^2 \widetilde{\pi}} {\strut\partial z_0 \partial z_2}\\
\dfrac{\strut\partial^2 \widetilde{\pi}}{\strut\partial z_1 \partial z_0}& \dfrac{\strut\partial^2 \widetilde{\pi}} {\strut\partial z_1 \partial z_1} &	 \dfrac{\strut\partial^2 \widetilde{\pi}} {\strut\partial z_1 \partial z_2}\\
\dfrac{\strut\partial^2 \widetilde{\pi}} {\strut\partial z_2 \partial z_0}& \dfrac{\strut\partial^2 \widetilde{\pi}}{\strut\partial z_2 \partial z_1} &\dfrac{\strut	 \partial^2 \widetilde{\pi}} {\strut\partial z_2 \partial z_2}
\end{bmatrix}$
=
$\begin{bmatrix}
6z_0&0&0\\
0&2&0\\
0&0&2
\end{bmatrix} $
\end{center}
which has nonzero determinants at $(\pm 1/ 27,0,0)$, and so the map $\widetilde{\pi}$ has two nondegenerate critical points and two vanishing cycles corresponding to the critical values. The regular fiber $\widetilde{F}=\widetilde{\pi}^{-1}(1)=\{z_0^3-\dfrac{z_0}{3^5}+z_1^2+z_2^2=1\}$ is the Brieskorn variety $V_1(3,2,2)$ ($4$-dimensional $A_2$-Milnor fiber) into which the fiber $F=V_1(3,2)$ of $\pi$ is properly embedded. Finally, we note that the intersection pattern of the vanishing cycles of both $\pi$ and  $\widetilde{\pi}$ is given in Figure \ref{fig:Intersection_pattern}-(a).
\end{example}

\begin{remark}
For any $n\geq 2$, one can generalize the above construction to obtain a compatible Lefschetz fibration on $D^{2n}$ equipped with its standard Stein structure: For a pair $(p,q)$,  consider
$$\pi:\mathbb{C}^n\rightarrow\mathbb{C}, \quad \pi(z_0,z_1,z_2,...,z_{n-1})=z_0^p+z_1^q+\displaystyle\sum_{i=2}^{n-1}z_i^2+M_\delta(z_0,z_1)$$ where $M_\delta(z_0,z_1)$ is a suitable morsification term so that all critical points are nondegenerate. Then the restriction map $\pi:D^{2n}\rightarrow D^2$ defines a compatible Lefschetz fibration on some $D^{2n}$ over a disk $D^2 \subset \mathbb{C}$. The regular fibers are diffeomorphic to the Brieskorn variety $V_\epsilon^{2n-2}(p,q,2,...,2)$, and the monodromy consists of the product of  $(p-1)(q-1)$ many right-handed Dehn twists. In particular, the binding of the induced open book is the Brieskorn manifold $\Sigma^{2n-3}(p,q,2,...,2)$.
\end{remark}
	

\section{Suspended Legendrian  2-links in the complement of Brieskorn manifolds}\label{lyon}
	 	
In \cite{plam}, Plamenevskaya extended the result of Lyon \cite{lyon}: For any Legendrian $1$-link $K$ in $(S^3,\xi^3_{std})$, there exists a torus link $\Sigma(p,q)$ such that $K$ is embedded in the interior of a Seifert surface of $\Sigma(p,q)$. In this section, we adapt these results for relative Legendrian $2$-links in $(S^5,\xi^5_{std})$. More precisely,
	 	
\begin{proposition}\label{prop3.1} Let $K=\sqcup_{j=1}^\ell K_j$ be a Legendrian $1$-link in the standard contact manifold $(S^3,\xi^3_{std})\subset (S^5,\xi^5_{std})$ such that each $ K_j$ is the boundary of  a Lagrangian $ 2 $-disk $  D^2 $ in $ D^4\subset \mathbb{C}^2 $ which is cylindrical near the contact boundary. Then there exists a Legendrian $2$-link $L=\sqcup_{j=1}^\ell L_j$ such that $(L,K)$ is a relative Legendrian $2$-link in the relative contact pair $(S^5,S^3)$ and a pair of Brieskorn manifolds $(\Sigma(p,q,2),\Sigma(p,q))$ contactly embedded in $(S^5,S^3)$ such that $\Sigma(p,q)$ is contactly and properly embedded in $\Sigma(p,q,2)$, and for any component $L_j$ of $L$ and its equator $K_j \subset K$, the following are satisfied:
\begin{enumerate}
			\item[(i)] $L_j$ is embedded in $int(V_\epsilon(p,q,2))$ where $V_\epsilon(p,q,2)\subset S^5$ is a Seifert filling of $\Sigma(p,q,2)$.
			\item[(ii)] $K_j$ is embedded in $int(V_\epsilon(p,q))$ where $V_\epsilon(p,q)\subset S^3$ is a Seifert surface of $\Sigma(p,q)$.
			\item[(iii)] The page framing of $L_j$ in $V_\epsilon(p,q,2)$ is equal to its contact framing in $(S^5,\xi^5_{std})$.
			\item[(iv)] The page framing of $K_j$ in $V_\epsilon(p,q)$ is equal to its contact framing in $(S^3,\xi^3_{std})$.
			\item[(v)] $V_\epsilon(p,q)$ is a properly embedded Weinstein subdomain of $V_\epsilon(p,q,2)$.
\end{enumerate}		
\end{proposition}
	 	
\begin{proof}  
	By \cite{plam}, we may assume there exists a torus knot $\Sigma(p,q)$ such that the equatorial Legendrian $1$-link $K\subset (S^3,\xi^3_{std})$ is embedded in the interior of a Seifert surface $F:= V_\epsilon(p,q)$ of $\Sigma(p,q)$. Also, since $F$ is the page of a compatible open book (with binding $\Sigma(p,q)$) on $(S^3,\xi^3_{std})$, the framing of each component of $K$ in $F$ is equal to its contact framing in $(S^3,\xi^3_{std})$ (see \cite{ao}). Therefore, (ii) and (iv) follow. By assumption, for each $j=1,...,\ell$, there exists an embedded Lagrangian disk $ D_j $ in $D^4\subset \mathbb{C}^2$ such that the boundary is the embedded Legendrian $K_j$ in $ S^3$. In what follows, we construct the Lagrangian $2$-knot $L_j$ in the Brieskorn variety $V_\epsilon(p,q,2)$ by lifting the disk $D_j$ under a double branched cover. We will also make use of the fact that $V_\epsilon(p,q,2)$ is obtained by suspending the Seifert surface $F=V_\epsilon(p,q)$ of $K$ as observed in the proof of Proposition \ref{prop2.1}. Consider the map $$ f: V_\epsilon(p,q,2)\to \mathbb{C}^2,\quad f(z_0,z_1,z_2)=(z_0,z_1) $$ where $(z_0,z_1,z_2)$ are complex coordinates on $\mathbb{C}^3$. This map is a two-fold cover branched along the fiber $\pi^{-1}(0)$ in $\mathbb{C}^2$ where $\pi$ is the Lefschetz fibration given in the proof of Proposition \ref{prop2.1}. For every point except on the fiber $\pi^{-1}(0)$, there are two preimage points of $f$ which are related by changing the sign on the $z_2$-coordinate. A disk that is not intersecting the fiber $\pi^{-1}(0)$ lifts to two disjoint embedded disks in $V_\epsilon(p,q,2)$. Since $f$ is a holomorphic covering map away from $\pi^{-1}(0)$, we may assume that $f$ is a symplectomorphism so that the two lifted disks are Lagrangian. If the disk has boundary equal to $K_j$ in $\pi^{-1}(0)$ and  a neighborhood of the boundary  of the disk projects to a line $[0,\epsilon)$ in $\mathbb{C}$ under $\pi$, then the two lifts of $D_j$ can be glued to form a smooth Lagrangian sphere in $V_\epsilon(p,q,2)$ whose equator projects to $(-\epsilon,\epsilon)$ under the Lefschetz fibration $\Pi$ on $V_\epsilon(p,q,2)$ also given in the proof of Proposition \ref{prop2.1}. Note that, if you take $(z^2\circ \Pi)$ or $(\Pi)^2$ is equal to $\pi\circ f$. Therefore, (i) follows.
	
	As explained in the proof of Proposition \ref{prop2.1}, $\widetilde{F}=V_\epsilon(p,q,2)$ and $F$ are the regular fibers of a compatible relative Lefschetz fibration on the Stein pair $(D^6,D^4)$, so $F$ is a Weinstein subdomain of $\widetilde{F}$ as claimed in (v). In particular, $\widetilde{F}$ is the page of the (induced) compatible open book on $(S^5,\xi^5_{std})$, and so its boundary (the binding) $\Sigma(p,q,2)$ is a contact submanifold of $(S^5,\xi^5_{std})$ and the binding $\Sigma(p,q)$ of the (induced) compatible open book on $(S^3,\xi^3_{std})$ is contactly embedded in $\Sigma(p,q,2)$.
	
	Finally, (iii) follows from Lemma \ref{lem:contact_page_framing} below which applies to all (odd) dimensions. This concludes the proof of Proposition \ref{prop3.1}.
\end{proof}

\begin{lemma} \label{lem:contact_page_framing}
 Let $L$ be any Legendrian $n$-sphere on a page $F^{2n}$ of a compatible open book on a contact manifold $(M^{2n+1},\xi=\emph{Ker}(\alpha))$. Then the page framing of $L$ coincides with its contact framing.
\end{lemma}

\begin{proof}
From compatibility we may assume that the Reeb vector field $R$ of $\alpha$ (defined by the equations $\alpha(R)=1$ and $\iota_Rd\alpha=0$) are transverse to both $\xi$ and each page of the open book, and also that $d\alpha$ restricts to a symplectic form on both $\xi$ and on $F$. Let $J$ (resp. $\tilde{J}$) be an almost complex structure on the symplectic bundle $\xi$ (resp. on the (exact) symplectic manifold $F$) compatible with $d\alpha|_{\xi}$ (resp. $d\alpha|_{F}$). Note at each point $p\in L$, we have $T_pL=\xi(p) \cap T_pF$. Observe $Jv \in \xi(p)/T_pL$ for every $v \in T_pL$ (otherwise $d\alpha|_{\xi}$ would not be symplectic). Therefore, the bundle $J(TL)\subset \xi|_{L}$ is a subbundle of the normal bundle $N(L)$ of $L$ in $M$. $N(L)$ can be identified (by Legendrian neighborhood theorem) with $T^*L \oplus \epsilon$ where $\epsilon$ corresponds to the summand generated by the restriction $R|_{L}$ of the Reeb field. Since $L$ is an embedded $n$-sphere, we obtain  $$N(L)= J(TL)\oplus \langle R|_{L} \rangle \cong T^*L \oplus \epsilon \cong T^*S^n \oplus \epsilon \cong S^n \times \mathbb{R}^{n+1}.$$ That is, a compatible almost complex structure $J$ on $(\xi,d\alpha|_{\xi})$ determines a trivialization $\Xi_J: S^n \times \mathbb{R}^{n+1} \to N(L)= J(TL)\oplus \langle R|_{L} \rangle$ which is known as the \textit{contact framing} of $L$.

On the other hand, being Legendrian on a page of a compatible open book, $L$ is Lagrangian in $(F,d\alpha|_{F})$. So, at each $p\in L$, $\tilde{J}v \in T_pF/T_pL$ for every $v \in T_pL$ (as $d\alpha|_{TL}=0$). Thus, the bundle $\tilde{J}(TL)\subset TF|_{L}$ is the normal bundle of $L$ in $F$ which can be identified (by Lagrangian neighborhood theorem) with $T^*L$. Similarly, considering $\epsilon$ as the bundle (of rank 1) generated by $R|_{L}$ (which is transverse to $F$), we obtain $$N(L)= \widetilde{J}(TL)\oplus \langle R|_{L} \rangle \cong T^*L \oplus \epsilon \cong T^*S^n \oplus \epsilon \cong S^n \times \mathbb{R}^{n+1}.$$ That is, a compatible almost complex structure $\widetilde{J}$ on $(F,d\alpha|_{F})$ determines a trivialization $\Xi_{\widetilde{J}}: S^n \times \mathbb{R}^{n+1} \to N(L)= \widetilde{J}(TL)\oplus \langle R|_{L} \rangle$ which is known as the \textit{page framing} of $L$.

It remains to show that the maps $\Xi_{J},\Xi_{\widetilde{J}}:S^n \times \mathbb{R}^{n+1} \to N(L)\subset M$ are isotopic as maps into $M$. Let $N$ be a closed neighborhood of $L$ in the page $F$. Consider the symplectic bundles $E_0=\xi|_{N}\to N$, $E_1=TF|_{N}\to N$ whose symplectic forms both descend from $d\alpha$. Compatibility condition is equivalent to the fact that, via a suitable isotopy through oriented $2n$-plane fields on the (compact) subset $N\subset F$, the fibers of $E_0$ can be made arbitrarily close to the corresponding fibers of $E_1$. Thus, there exists a homotopy of oriented $2n$-plane bundles, say $E_t \to N$ ($t\in [0,1]$), and its restriction $E_t|_{L} \to L$ defines a homotopy of oriented $2n$-plane bundles between $\xi|_{L}\to L$ and $TF|_{L}\to L$. Since each bundle is oriented and $d\alpha$ is a volume form on fibers when $t=0,1$, it is a volume form on fibers of $E_t|_{L} \to L$ for all $t\in [0,1]$, and so is of full rank which means that the Reeb field $R$ of $\alpha$ is transverse to the fibers of $E_t|_{L} \to L$ for all $t$. Also since $TL=\xi|_{L}\cap TF|_{L}$, we may assume that $TL$ is fixed during this homotopy. Thus, we have a homotopy of oriented $2n$-plane bundles $J_t \to L$ ($t\in [0,1]$) between the normal summands $J_0=J(TL)$ and $J_1=\widetilde{J}(TL)$ of $\xi|_{L}$ and $TF|_{L}$, respectively. Using (following) $J_t$ and the transversality of $R$ to the fibers, we conclude that there exists an isotopy $$\Xi_t: S^n \times \mathbb{R}^{n+1} \to N(L)= J_t\oplus \langle R|_{L} \rangle, \quad t\in [0,1]$$ between $\Xi_0=\Xi_J$ and $\Xi_1=\Xi_{\widetilde{J}}$ as claimed.
 			 		
\end{proof}

\begin{definition}
The pair $(L,K)\subset (S^5,S^3)$ constructed in Proposition \ref{prop3.1} will be called a \textit{suspended $2$-link based on K}.
\end{definition}

Clearly, every suspended Legendrian $2$-link is a relative Legendrian $2$-link by its construction. However, whether every relative Legendrian $2$-link in $(S^5,S^3)$ is Legendrian isotopic to a suspended one is not clear in general.  We introduce:

\begin{definition}\label{def:suspendible_link}
A relative Legendrian $2$-link $(L,K)$ in the standard contact pair $(S^5,S^3)$ is said to be \emph{suspendible} if it is Legendrian isotopic to a suspended Legendrian $2$-link based on some Legendrian $1$-link $K'$ (possibly different than $K$) in the standard contact $S^3$.
\end{definition}

\begin{remark} \label{rmk:subcritical}
For the purpose of the present paper, one needs to define suspendibility of a relative Legendrian $2$-link $(L,K)$ embedded in the contact boundary of a subcritical Stein $6$-domain. If $\widetilde{W}$ is such a domain, then it can be written as a product $\widetilde{W}=\widetilde{X} \times D^2$ for some Stein $4$-domain $\widetilde{X}$, and the pair $(\widetilde{W},\widetilde{X})$ is a relative Stein pair because it clearly satisfies the conditions (0)-(2) of Definition \ref{def:Stein_Pair}. (The Weinstein handles of $\widetilde{W}$ can be taken as the thickening of those of $\widetilde{X}$). Note that using a given Weinstein handle decomposition of $\widetilde{X}$, $(\widetilde{W},\widetilde{X})$ can be described by a relative Stein diagram in the dotted form containing only dotted unknots ($1$-handles of $\widetilde{X}$, and so $\widetilde{W}$) and the framed dashed knots ($2$-handles of $\widetilde{X}$, and so $\widetilde{W}$). One should note that such  a diagram is not unique since one can start with a different Weinstein handle decomposition of $\widetilde{X}$ (still defining the same Stein structure on $\widetilde{X}$ (and so on $\widetilde{W}$) up to Stein deformation). So, let $RSD_{\mathcal{H}}(\widetilde{X}\times D^2)$ be the relative Stein diagram (in the dotted form) of $\widetilde{W}=\widetilde{X} \times D^2$ obtained from a given Weinstein handle decomposition $\mathcal{H}$ of $\widetilde{X}$. Recall that, in a diagram $RSD_{\mathcal{H}}(\widetilde{X}\times D^2)$, replacing dotted unknots with $(+1)$-framed unknots, we obtain a relative contact surgery diagram, denoted by $CSD_{\mathcal{H}}(\partial(\widetilde{X}\times D^2))$, describing the contact boundary $\partial(\widetilde{X}\times D^2)$. Therefore, one can realize $(L,K)$ in this surgery diagram by drawing the components of $K$ as solid knots. Forgetting framings on the $(+1)$-framed unknots and the framed dashed knots in the diagram, we obtain a Legendrian $1$-link $K_{\mathcal{H}}$ (in the standard contact $S^3$) which consists of the above (unframed) unknots and (unframed) dashed knots. Note that after forgetting all the framings in the diagram $CSD_{\mathcal{H}}(\partial(\widetilde{X}\times D^2))$, the link $(L,K)$ becomes a relative Legendrian $2$-link in the standard contact pair $(S^5,S^3)$.
\end{remark}

Under the setting and notation given in Remark \ref{rmk:subcritical}, we define

\begin{definition}\label{def:suspendible_link_general}
A relative Legendrian $2$-link $(L,K)$ in the contact boundary $\partial(\widetilde{X}\times D^2)$ of a subcritical Stein $6$-domain $\widetilde{W}=\widetilde{X}\times D^2$ is said to be \emph{suspendible} if there exists a Weinstein handle decomposition $\mathcal{H}$ of $\widetilde{X}$ such that there exists a Legendrian isotopy in the complement $S^5\setminus K_{\mathcal{H}}$ from $(L,K)$ to a suspended Legendrian $2$-link based on some Legendrian $1$-link $K'$ (possibly different than $K$) in the complement $S^3\setminus K_{\mathcal{H}}$.
\end{definition}


\section{Proof of Theorem \ref{thm:Main_theorem}} \label{proof}
 	
We will first explain that any admissible Stein pair can be described by a relative Stein diagram as promised in the introduction. Then we will prove the theorem in three cases/steps: Proposition \ref{mainprop1} (lack of $2$-handles), Proposition \ref{mainprop2} (lack of $3$-handles), and the general case.

 Let $(W^6,X^4)$ be an admissible Stein pair. By definition they admit a relative handle  decomposition such that the link $L$ (along which the $3$-handles of $W$ are attached) and its equatorial link $K$ (along which the corresponding $2$-handles of $X$ are attached) together form a relative Legendrian $2$-link $(L,K)$ in the contact boundary of the Stein $6$-domain
$$W_2=D^4 \cup \{\textrm{$1$-handles of $W$}\} \cup \{\textrm{$2$-handles of $W$}\}$$ which splits as $W_2=\widetilde{X}\times D^2$ where $\widetilde{X}$ is a Stein $4$-domain. By assumption $(L,K)\subset \partial(\widetilde{X}\times D^2)$ is suspendible, therefore, there exists a Weinstein handle decomposition $\mathcal{H}$ of $\widetilde{X}$ such that we can isotope $(L,K)$ (through Legendrian $2$-links in $S^5\setminus K_{\mathcal{H}}$) to some suspended Legendrian $2$-link. Assuming Legendrian isotopy taking $(L,K)$ to a suspended Legendrian $2$-link has been already performed, we may assume that $(L,K)$ is suspended.

In the diagram $RSD_{\mathcal{H}}(\widetilde{X}\times D^2)$, suppose the $2$-handles of $\widetilde{X}$ (and so $\widetilde{W}$, and hence $W$) are attached along an isotropic (Legendrian) $1$-link $M$. The $1$-link $M$ and the suspended $2$-link $(L,K)$ lie in the relative contact pair $(\#_r S^1 \times S^4, \#_r S^1 \times S^2)$ (equipped with $(\eta^5_r,\eta^3_r)$) which is the relative boundary of the relative Stein pair $(\natural_r S^1 \times D^5, \natural_r S^1 \times D^3)$ where $r$ is the number of $1$-handles of $\widetilde{X}$ (and so $\widetilde{W}$, and hence $W$). Thus, by first introducing $r$ dotted unknots and then drawing the components of $M$ (as dashed knots) and the components of $K$ (as solid knots), one can form a relative Stein diagram of $(W,X)$ as depicted in Figure \ref{fig:Dotted_rel_Stein_diagram}.

One should note that the resulting relative Stein diagram describes the admissible Stein pair $(W,X)$ uniquely up to Stein deformation. This is because of the fact that the construction (given in Proposition \ref{prop3.1}) of a suspended Legendrian $2$-link depends only on the base knot, and that the suspended link $(L,K)$ is Legendrian isotopic to the original relative Legendrian $2$-link along which the relative $3$-handles of $W$ are attached. Hence, we have shown that

\begin{lemma} \label{lem:admissible_to_diagram}
Every admissible Stein pair can be described via some relative Stein diagram. Conversely, every relative Stein diagram determines an admissible Stein pair which is unique up to Stein deformation.
\end{lemma}

\begin{proposition}\label{mainprop1}
Every admissible Stein pair without relative $2$-handles admits a compatible relative Lefschetz fibration.
\end{proposition}

\begin{proof}
The relative Stein diagram of an admissible Stein pair $(W,X)$ without relative $2$-handles consists of $r$ dotted unknots $U_1, ..., U_r$ (representing relative $1$-handles) and the solid knots drawn for each component of $K$ (the base of the suspended Legendrian $2$-link $(L,K)$ along which the relative $3$-handles are attached) as discussed above. (Here, in particular, it is assumed that this relative Stein diagram of $(W,X)$ is obtained from a diagram $RSD_{\mathcal{H}}(\widetilde{X}\times D^2)$ for a suitable Weinstein handle decomposition $\mathcal{H}$.) First, we forget the dots on $U_i$'s and consider the entire diagram as a relative $2$-link in the relative standard contact pair $(S^5,S^3)$. Then, by Proposition \ref{prop3.1}, there exists a pair of  Brieskorn manifolds $(\Sigma(p,q,2),\Sigma(p,q))$ in $(S^5,S^3)$ such that $(L,K)$ is properly embedded on their Seifert fillings $(V_\epsilon(p,q,2),V_\epsilon(p,q))$ as Legendrian. Note that each $U_i$ is also realized on $V_\epsilon(p,q)\subset V_\epsilon(p,q,2)$. Moreover, by Proposition \ref{prop2.1}, there exist Lefschetz fibrations on $D^6$ and $D^4$ with regular fibers $F_{D^6}=V_\epsilon(p,q,2)$ and $F_{D^4}=V_\epsilon(p,q)$, respectively, as depicted in Figure \ref{fig:relative_fibers}-(a). Recall (from Remark \ref{rmk:Compatible_Rel_Lefs_fib}) that the monodromies of these fibrations are denoted by $\widetilde{\varphi}_{p,q}$ and $\varphi_{p,q}$, respectively, which are the product of $(p-1)(q-1)$ many right-handed Dehn twists along explicit spheres.\\

Next, generalizing the idea used for $D^4$ in \cite{ao} to $D^6$, we will extend the above compatible relative Lefschetz fibration on $(D^6,D^4)$ to a compatible relative Lefschetz fibration on the relative Stein pair $(D^6\cup 1$-handles, $D^4\cup 1$-handles) as follows:\\
	
	First, isotope each $U_i$ near the binding $\Sigma(p,q)$ of the induced open book on the boundary $S^3=\partial D^4$ as in \cite{ao} so that it becomes transverse to each page $F_{D^4}$. Once this is achieved, $U_i$ will also transversally intersect (near the binding $\Sigma(p,q,2)$) each page $F_{D^6}$ of the induced open book on the boundary $S^5=\partial D^6$ (thanks to the relative Lefschetz fibration structure). For each $i=1,..., r$, attaching a $1$-handle to $D^4$ corresponds to first pushing the interior of the Lagrangian spanning disk of $U_i$ into $D^4$ and then removing a tubular neighborhood of this resulting disk. This is equivalent to puncture each fiber $F_{D^4}$. Since each fiber $F_{D^4}$ is properly embedded in the corresponding fiber $F_{D^6}$, this will correspond to carving a small neatly embedded  Lagrangian  $2$-disk in the fiber $F_{D^6}$. On the other hand, by the condition (1) in Definition \ref{def:Stein_Pair}, the cocore of each $1$-handle of $X$ is properly embedded in the cocore of the corresponding $1$-handle of $W$. Therefore, simultaneously carving these small neatly embedded Lagrangian $2$-disks from the fibers $F_{D^6}$ corresponds to attaching the corresponding $6$-dimensional relative $1$-handle of $W$ to $D^6$. In other words, attaching each relative $1$-handle of $(W,X)$ to the Stein pair $(D^6,D^4)$ is equivalent to carving a small Lagrangian disk pair $(D^2,D^0)$ from each fiber pair $(F_{D^6},F_{D^4})=(V_\epsilon(p,q,2),V_\epsilon(p,q))$. This is depicted in Figure \ref{fig:relative_fibers}-(b) where a single relative $1$-handle is considered (for simplicity), and the ``punctured'' fibers (i.e., new fibers after carving disks) are denoted by $\mathring{F}_{D^6}$ and $\mathring{F}_{D^4}$.\\
	
	After attaching all relative $1$-handles of $(W,X)$ to $(D^6,D^4)$, we get a new compatible relative Lefschetz fibration on $(D^6\cup 1$-handles, $D^4\cup 1$-handles)$=(\natural_r S^1 \times D^5, \natural_r S^1 \times D^3)$ whose regular fibers are  $r$ times ``punctured'' $(V_\epsilon(p,q,2),V_\epsilon(p,q))$ near boundaries, which we denote by $(\mathring{F}_{D^6},\mathring{F}_{D^4})=(\mathring{V}_\epsilon(p,q,2),\mathring{V}_\epsilon(p,q)),$ and whose monodromy is obtained by extending $\widetilde{\varphi}_{p,q}$ and $\varphi_{p,q}$ (over the $1$-handles) by identity.\\
	
	Note that the suspended Legendrian $2$-link $(L,K)$ still sits on a relative page of the compatible relative open book on the boundary contact pair $(\#_r S^1 \times S^4, \#_r S^1 \times S^2)$ so that the framing conditions are still satisfied for the components of both $L$ and $K$ (as stated in Proposition \ref{prop3.1}). 
	
\begin{figure}[H]
			\centering
			\includegraphics[scale=1.3]{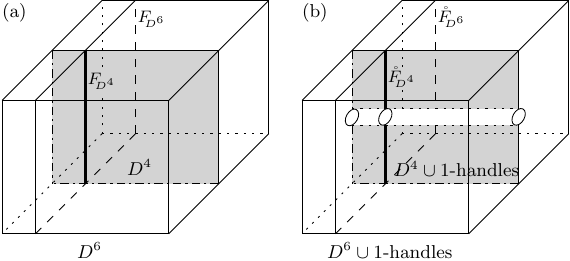}
			\vspace{.3cm}
			\caption{(a) Relative fibers $(F_{D^6},F_{D^4})$ of the Stein Pair $(D^6,D^4)$, (b) Punctured relative fibers $(\mathring{F}_{D^6},\mathring{F}_{D^4})$ of the Stein Pair $(D^6\cup 1$-handles, $D^4\cup 1$-handles).}
			\label{fig:relative_fibers}
\end{figure}

Now assume that $(L,K)$ has $\ell$ components, say $(L_j,K_j)$, $j=1,..., \ell$. So, for each $j$, we attach a relative  $3$-handle $({H}_j,{h}_j)$ to $(\natural_r S^1 \times D^5, \natural_r S^1 \times D^3)$ along the relative Legendrian $2$-knot $(L_j,K_j)$ which is properly embedded in $(\mathring{F}_{D^6},\mathring{F}_{D^4})$. Each ${H}_j$ (resp. ${h}_j$) is attached along a Legendrian $L_j$ (resp. $K_j$) which lies on a Weinstein (so symplectic) page $\mathring{F}_{D^6}$ (resp. $\mathring{F}_{D^4}$) of the boundary (compatible) open book so that the contact framing of $L_j$ (resp. $K_j$) coincides with its page framing on $\mathring{F}_{D^6}$ (resp. on $\mathring{F}_{D^4}$). Furthermore, recall that any Legendrian submanifold sitting on a page of a compatible open book is Lagrangian. Thus, we conclude that ${H}_j$ (resp. ${h}_j$) is a six (resp. four) dimensional Lefschetz handle. Consequently, attaching each relative $3$-handle corresponds to attaching a ``\textit{relative}'' Lefschetz handle. So, attaching all such handles constructs a compatible relative Lefschetz fibration, say $(\pi_W,\pi_X)$, on  the admissible Stein pair $(W,X)$. Note that the fibers of $(\pi_W,\pi_X)$ are still $(\mathring{F}_{D^6},\mathring{F}_{D^4})$. On the other hand, the monodromy of $(\pi_W,\pi_X)$ is the pair $$(\varphi_W,\varphi_X)=(\widetilde{\varphi}_{p,q}\cdot\tau_{L_1}\tau_{L_2}\cdots\tau_{L_\ell},\;\;\varphi_{p,q}\cdot\tau_{K_1}\tau_{K_2}\cdots\tau_{K_\ell})$$ where $\tau_{L_j}$ (resp, $\tau_{K_j}$) is the right-handed Dehn twist along $L_j$ (resp. $K_j$).

\end{proof}

\begin{proposition}\label{mainprop2}
	 On every subcritical Stein $ 6 $-domain (no handles of critical index) there exists a relative Stein pair structure which admits a nontrivial compatible relative Lefschetz fibration with regular fibers  $\mathring{V}_\epsilon(p,q,2)\cup\bigcup\limits_{j=1}^{2k} \tilde{h_j}$ for some $p,q$, where each $\tilde{h_j}$ is a $4$-dimensional $2$-handle.
\end{proposition}

\begin{proof}
Let $\widetilde{W}$ be a subcritical Stein $6$-domain. By the result of  Cieliebak \cite{cieliebak}, there is a splitting $\widetilde{W}=\widetilde{X}\times D^2$ where $\widetilde{X}$ is a Stein $4$-domain. Starting from a Weinstein handle decomposition $\mathcal{H}$ of $\widetilde{X}$, one obtains a compatible Lefschetz fibration on $\widetilde{X}$ with fibers $r$ times ``punctured'' $V_\epsilon (p,q)$ for some $p,q$ (where $r$ is the number of $1$-handles of $\widetilde{X}$ in $\mathcal{H}$), and whose monodromy is that of the torus link $\Sigma(p,q)$ composed with right-handed Dehn twists corresponding to the Stein (Lefschetz) $2$-handles of $\widetilde{X}$ in $\mathcal{H}$, as constructed in \cite{ao}. So, we have a (Lefschetz) fibration map $\pi_\mathcal{H}: \widetilde{X}\rightarrow D^2$ which has finitely many nondegenerate critical points.\\

We can extend $\pi_\mathcal{H}$ to a fibration map on $\widetilde{X}\times D^2$ by using the  ``\textit{suspension}'' technique, as in Proposition \ref{prop2.1}. More precisely, consider the following map

\begin{center}
	$\widetilde{\pi}_\mathcal{H}:\widetilde{X}\times D^2 \to  D^2$, \quad
	$ (p,z_2)\mapsto \pi_\mathcal{H}(p)+z_2^2 $.
\end{center}

Critical points of $\widetilde{\pi}_\mathcal{H}$ are given by the local equations
$${\frac{\partial \widetilde{\pi}_\mathcal{H}}{\partial z_0}  }={\frac{\partial {\pi}_\mathcal{H}} {\partial z_0} }=0, \quad {\frac{\partial \widetilde{\pi}_\mathcal{H}}{\partial z_1}}={\frac{\partial {\pi}_\mathcal{H} }{\partial z_1}}=0, \quad {\frac{\partial \widetilde{\pi}_\mathcal{H}}{\partial z_2} }=0$$
where $(z_0,z_1)$ are the local complex coordinates around a critical point in $\widetilde{X}$. So, the critical points of $\widetilde{\pi}_\mathcal{H}$ are the same as those of $\pi_\mathcal{H}$ with the additional zero $z_2$-coordinate. Also, if we check the Hessian of $\widetilde{\pi}_\mathcal{H}$ at each critical point, we see that it is nondegenerate because the corresponding Hessian of $\pi_\mathcal{H}$ is nondegenerate. Each regular fiber of $\widetilde{\pi}_\mathcal{H}$ can be seen as double branched cover of $\widetilde{X}$ which is branched over ${\pi}_\mathcal{H}^{-1}(c)$ where $c$ is a regular value of $\widetilde{\pi}_\mathcal{H}$. With this point of view, the effect of attaching $k$-many $4$-dimensional $2$-handles to the fibers of $\widetilde{\pi}_\mathcal{H}$ will result in attaching $2k$-many $2$-handles to $r$-times punctured $V_\epsilon(p,q,2)$. The monodromy of $\widetilde{\pi}_\mathcal{H}$ is the product of right-handed Dehn twists along the spheres obtained by taking the suspensions of the corresponding vanishing cycles of $\pi_\mathcal{H}$. Therefore, $\widetilde{\pi}_\mathcal{H}$ is a compatible Lefschetz fibration map on $\widetilde{W}$ with nontrivial monodromy.\\

Moreover, the pair $(\widetilde{W},\widetilde{X})$ is a relative Stein pair since it fulfills the conditions of Definition \ref{def:Stein_Pair}. To see this, observe that Stein domains $\widetilde{W}$ and $\widetilde{X}$ both share the same handle decomposition. Indeed, the cocores of index $1$- and $2$-handles of $\widetilde{X}$ are properly embedded in those of the corresponding index $1$ and $2$-handles of $\widetilde{W}$ (as the latter are just the thickenings of the former). Note that, the condition $(3)$ is obviously satisfied since there are no index $3$-handles. Hence, we conclude that the pair $(\widetilde{\pi}_\mathcal{H},\pi_\mathcal{H})$ is a compatible relative Lefschetz fibration on the relative (admissible indeed) subcritical Stein pair $(\widetilde{W},\widetilde{X})$.

\end{proof}

Finally, we discuss the most general case as it has been stated in Theorem \ref{thm:Main_theorem}.

\begin{proof}[Proof of Theorem \ref{thm:Main_theorem}]
Let $(W,X)$ be any admissible Stein pair. Then we have
\begin{center}
	$W=\widetilde{W}\cup (\sqcup_{j=1}^{\ell} H_j)$ \quad and \quad $X=\widetilde{X}\cup (\sqcup_{j=1}^{\ell} h_j)$
\end{center}
where $\widetilde{W}$ is the subcritical part of $W$ which split as $\widetilde{W}=\widetilde{X}\times D^2$ for some Stein $4$-domain $\widetilde{X}$, and $(H_j,h_j)$ are the relative $3$-handles of $W$ attached along a relative Legendrian $2$-link $(L,K)$. By Proposition \ref{mainprop2}, there exists a compatible relative Lefschetz fibration $(\widetilde{\pi}_\mathcal{H},\pi_\mathcal{H})$ on the subcritical part $(\widetilde{W},\widetilde{X})$ of $(W,X)$ for any Weinstein handle decomposition $\mathcal{H}$ of $\widetilde{X}$. We choose $\mathcal{H}$ in such a way that the link $(L,K)$ is Legendrian isotopic (in $S^5 \setminus K_{\mathcal{H}}$) to a suspended Legendrian $2$-link. Such a decomposition $\mathcal{H}$ exists by admissibility (i.e., by suspendibility of $(L,K)$, see Definition \ref{def:suspendible_link_general}). Hence, assuming the above Legendrian isotopy has been already performed, we may assume $(L,K)$ is a suspended Legendrian $2$-link sitting (as Lagrangian) on a relative page $(\widetilde{F}_{\mathcal{H}},F_{\mathcal{H}})$ of the relative compatible open book (induced by $(\widetilde{\pi}_\mathcal{H},\pi_\mathcal{H})$) on the boundary relative contact pair $(\partial \widetilde{W},\partial \widetilde{X})$. Note that by Lemma \ref{lem:contact_page_framing} the contact and the page framing of each component $(L_j,K_j)$ of $(L,K)$ coincide.

To finish the proof, one just needs to proceed as in the last part of the proof of Proposition \ref{mainprop1}. More precisely, we have observed that each relative $3$-handle $(H_j,h_j)$  attached along $(L_j,K_j)$ is a relative Lefschetz handle, i.e.,  $(\widetilde{\pi}_\mathcal{H},\pi_\mathcal{H})$ extends over each $(H_j,h_j)$. Thus, attaching all $(H_j,h_j)$'s, we get a compatible relative Lefschetz fibration, say $(\pi_W,\pi_X)$, on  the admissible Stein pair $(W,X)$ with relative regular fibers $(\widetilde{F}_{\mathcal{H}},F_{\mathcal{H}})$. If $(\widetilde{\varphi}_{\mathcal{H}},\varphi_{\mathcal{H}})$ denotes the monodromy of $(\widetilde{\pi}_\mathcal{H},\pi_\mathcal{H})$, then the monodromy of $(\pi_W,\pi_X)$ is the pair $$(\varphi_W,\varphi_X)=(\widetilde{\varphi}_\mathcal{H}\cdot\tau_{L_1}\tau_{L_2}\cdots\tau_{L_\ell},\;\;\varphi_\mathcal{H}\cdot\tau_{K_1}\tau_{K_2}\cdots\tau_{K_\ell})$$ where $\tau_{L_j}$ (resp, $\tau_{K_j}$) denotes the right-handed Dehn twist along $L_j$ (resp. $K_j$) as before.

\end{proof}
 		
	
\section{Examples}
\begin{example}
Let $D(T^\ast S^n)$ denote the standard closed unit disk bundle over the $n$-sphere. The pair $(D(T^\ast S^3),D(T^\ast S^2))\subset (\mathbb{C}^4,\mathbb{C}^3)$ equipped with their \textit{standard} complex structures (inherited from complex spaces) is an admissible relative Stein pair: It can be obtained by attaching a relative  $3$-handle to the trivial Stein pair $(D^6,D^4)$ along a relative Legendrian $2$-unknot $(L,K)\subset (S^5,S^3)$ embedded in the simplest possible way. So, it admits a relative Weinstein handle decomposition with no relative $1$- and $2$-handles, and clearly we may assume $(L,K)$ to be suspended. Therefore, following the steps in Proposition \ref{mainprop1} (we take $p=q=2$), we construct a compatible relative Lefschetz fibration on the admissible relative Stein pair $(D(T^\ast S^3),D(T^\ast S^2))$ with a relative (regular) fiber $(V_1(2,2,2), V_1(2,2))\cong (D(T^\ast S^2),D(T^\ast S^1))$ and the monodromy
$$(\widetilde{\varphi}_{2,2},\varphi_{2,2}) =(\tau_{S^2},\tau_{S^1})$$
where the right-handed Dehn twists are along the zero sections $S^2$ and $S^1$ on the fibers $D(T^\ast S^2)$ and $D(T^\ast S^1 )$, respectively. The relative Stein diagram of $(D(T^\ast S^3),D(T^\ast S^2))$ used in the above construction is given in Figure \ref{fig:cotangent_bundles}.
		
\begin{figure}[h!]
			\centering
			\includegraphics[scale=0.5]{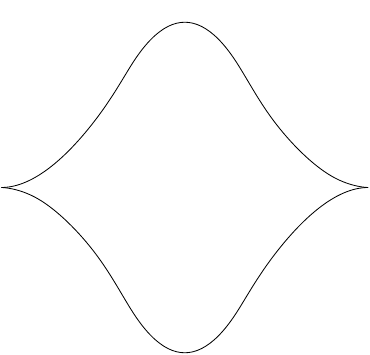}
			\caption{A relative Stein diagram for the standard $(D(T^\ast S^3),D(T^\ast S^2))$.}
			\label{fig:cotangent_bundles}
\end{figure}	
		
Indeed, one can generalize the above construction: If we consider $k$-many linear plumbings of the pair $(D(T^\ast S^3),D(T^\ast S^2))$, we get the pair $(A_k^6,A_k^4)\cong (V_1(k+1,2,2), V_1(k+1,2))$, which is also an admissible relative Stein pair (with their \textit{standard} complex structures inherited from complex spaces in which they live). It admits a relative Stein diagram as in Figure \ref{fig:linear_plumbing} where a relative $3$-handle is attached along the suspended $2$-unknot $(L_i,K_i)$ for each $i=1,..., k$. (No relative $1$- and $2$-handles again.) Then one easily construct a compatible relative Lefschetz fibration on the admissible relative Stein pair $(A_k^6,A_k^4)$ with the (relative) fiber $(V_1(2k,2k,2), V_1(2k,2k))$ and the monodromy $(\widetilde{\varphi}_{2k,2k}\cdot \tau_{L_1}\cdots \tau_{L_k} ,\varphi_{2k,2k}\cdot \tau_{K_1}\cdots \tau_{K_k})$ where $\tau_{L_i},\tau_{K_i}$ are the right-handed Dehn twists along the zero sections of the pair $(D(T^\ast S^3),D(T^\ast S^2))$ plumbed in the $i$th place.
\end{example}		
		
\begin{figure}[h!]
			\centering
			\includegraphics[scale=0.5]{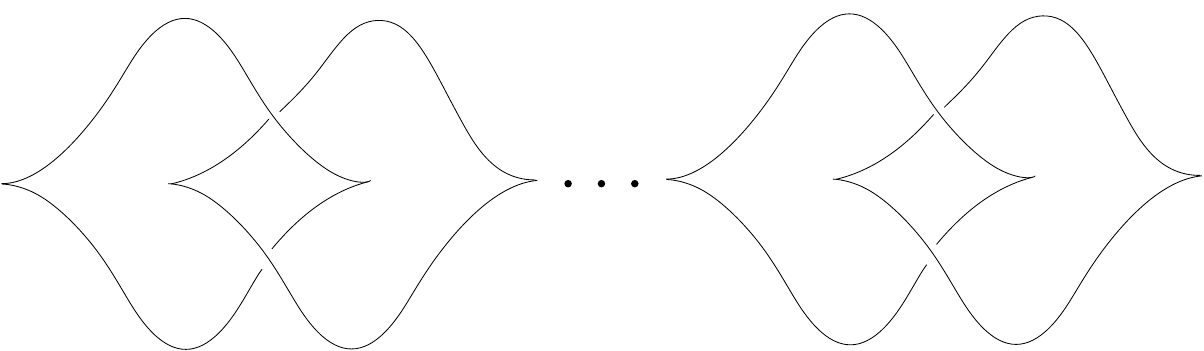}
			\caption{A relative Stein diagram for the standard $(A_k^6,A_k^4)$.}
			\label{fig:linear_plumbing}
\end{figure}

\begin{example}
Consider the admissible relative Stein pair $(W,X)$ determined uniquely (up to Stein deformation) by the relative Stein diagram in Figure \ref{fig:fishtail fiber}-(a). So the six-dimensional Stein domain $W$ consists of (Stein) handles of each index (up to 3): One $0$-handle, two $1$-handles, one $2$-handle (attached along the isotropic knot $K$ drawn dashed) and one $3$-handle (attached along the suspended (relative) Legendrian $2$-knot $\widetilde{L}$ based on $\widetilde{K}$ drawn (bold) solid). The four-dimensional Stein subdomain $X$ is described by the Stein diagram in Figure \ref{fig:fishtail fiber}-(b) which is, indeed, the closed neighborhood of a fishtail fiber in an elliptic fibration (given in Figure \ref{fig:fishtail fiber_2}). The subcritical part of $(W,X)$ is the admissible Stein pair $(\widetilde{W},\widetilde{X})$ which splits as $(\widetilde{X}\times D^2, \widetilde{X})$ where $\widetilde{X}$ is the $D^2$-bundle over $T^2$ with Euler number 0 (i.e., $\widetilde{X}\cong D^2 \times T^2$, so $\widetilde{W}\cong D^4 \times T^2$). For a suitable choice of $(p,q)$, we can construct (as in the proof of Proposition \ref{mainprop2}) a compatible relative Lefschetz fibration on the subcritical part $(\widetilde{W},\widetilde{X})$ with fibers twice ``punctured'' Brieskorn varieties $(V_\epsilon(p,q,2),V_\epsilon(p,q))$ (containing both $(\widetilde{L},\widetilde{K}$) and $(L,K)$ where $L$ is the suspension of $K$) and the monodromy $$(\widetilde{\varphi}_{p,q}\cdot \tau_{L},\varphi_{p,q }\cdot \tau_{K})$$
where $\tau_{K}$, $\tau_{L}$ are the right-handed Dehn twists along $K$ and $L$, respectively.\\

\begin{figure}[h!]
	\centering
	\includegraphics[scale=0.8]{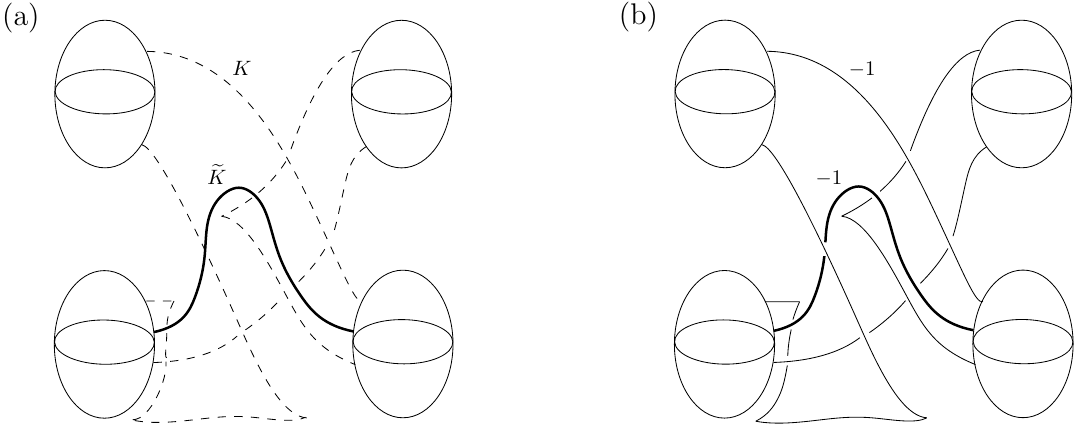}
	\caption{(a) A relative Stein diagram for $(W,X)$, (b) The Stein $4$-subdomain $X$ (the closed neighborhood of a fishtail fiber in an elliptic fibration).}
	\label{fig:fishtail fiber}
\end{figure}

The admissible relative Stein pair $(W,X)$ is obtained by attaching relative $3$-handle to $(\widetilde{W},\widetilde{X})$ along $(\widetilde{L},\widetilde{K})$ over which the compatible relative Lefschetz fibration structure on $(\widetilde{W},\widetilde{X})$ extends. More precisely, $(W,X)$ also admits a compatible relative Lefschetz fibration with the fibers twice ``punctured'' Brieskorn varieties $(V_\epsilon(p,q,2),V_\epsilon(p,q))$ and with the monodromy
$$(\widetilde{\varphi}_{p,q}\cdot \tau_{L}\cdot\tau_{\widetilde{L}},\varphi_{p,q }\cdot \tau_{K}\cdot\tau_{\widetilde{K}})$$ where $\tau_{\widetilde{K}}$, $\tau_{\widetilde{L}}$ are the right-handed Dehn twists along $\widetilde{K}$ and $\widetilde{L}$, respectively.
\end{example}

\newpage
\begin{figure}[h!]
	\centering
	\includegraphics[scale=0.8]{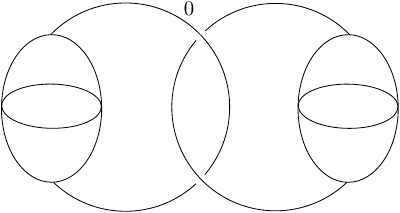}
	\caption{The closed neighborhood of a fishtail fiber in an elliptic fibration.}
	\label{fig:fishtail fiber_2}
\end{figure}

\medskip \noindent {\em Acknowledgments.\/} The authors would like to
thank Georgios Dimitroglou Rizell for helpful conversations and  comments on the first version.

	
	
	
	
\end{document}